\batchmode
\documentclass[12pt,oneside,english]{amsart}
\usepackage[T1]{fontenc}
\usepackage[latin9]{inputenc}
\usepackage{babel}
\usepackage{array}
\usepackage{amsthm}
\usepackage{amssymb}
\usepackage{graphicx}
\usepackage[unicode=true,pdfusetitle,
 bookmarks=true,bookmarksnumbered=false,bookmarksopen=false,
 breaklinks=false,pdfborder={0 0 1},backref=false,colorlinks=false]
 {hyperref}
\usepackage{breakurl}

\makeatletter

\providecommand{\tabularnewline}{\\}

\numberwithin{equation}{section}
\numberwithin{figure}{section}
\theoremstyle{plain}
\newtheorem{thm}{\protect\theoremname}
  \theoremstyle{plain}
  \newtheorem{prop}[thm]{\protect\propositionname}
  \theoremstyle{definition}
  \newtheorem{example}[thm]{\protect\examplename}
  \theoremstyle{plain}
  \newtheorem{lem}[thm]{\protect\lemmaname}
  \theoremstyle{plain}
  \newtheorem{cor}[thm]{\protect\corollaryname}
  \theoremstyle{remark}
  \newtheorem{rem}[thm]{\protect\remarkname}

\usepackage{ae}
\usepackage{aecompl}

\makeatother

  \providecommand{\corollaryname}{Corollary}
  \providecommand{\examplename}{Example}
  \providecommand{\lemmaname}{Lemma}
  \providecommand{\propositionname}{Proposition}
  \providecommand{\remarkname}{Remark}
\providecommand{\theoremname}{Theorem}

\begin{document}

\title{Antichain cutsets of strongly connected posets}

\author{Stephan Foldes and Russ Woodroofe}

\thanks{\hspace{-0.5cm}%
\begin{tabular}{p{0.76\columnwidth}>{\centering}p{0.1\columnwidth}}
Acknowledgement: The work of the first-named author has been co-funded
by Marie Curie Actions, and supported by the National Development
Agency (NDA) of Hungary and the Hungarian Scientific Research Fund
(OTKA) within a project hosted by the University of Miskolc, Department
of Analysis. The work was also completed as part of the TAMOP-4.2.1.B-10/2/KONV-2010-0001
project at the University of Miskolc, with support from the European
Union, co-financed by the European Social Fund.  &
\vspace{-0.2cm}
$\quad\,\,\,\,$\includegraphics[height=0.9cm]{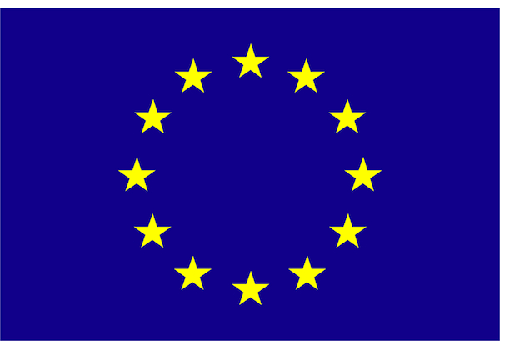}\vspace{0.1cm}
\includegraphics[height=1.2cm]{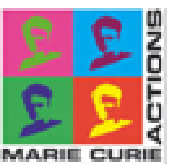}$\quad$\includegraphics[height=1.2cm]{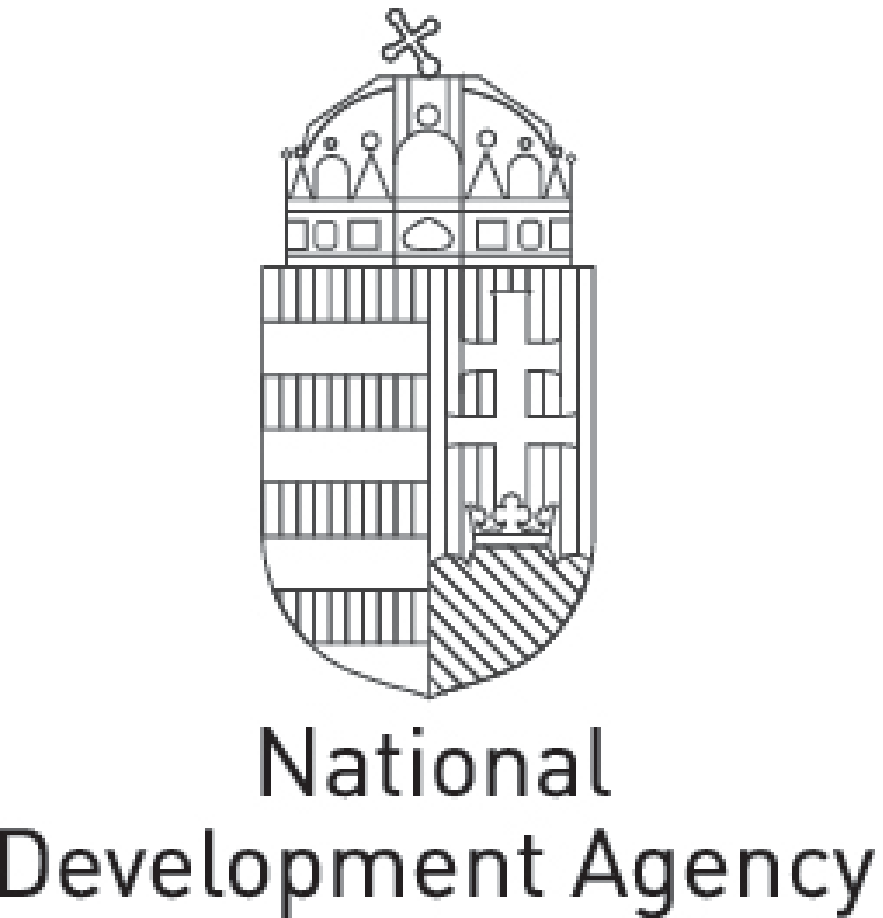}\\
\vspace{0.2cm}
~\includegraphics[height=1cm]{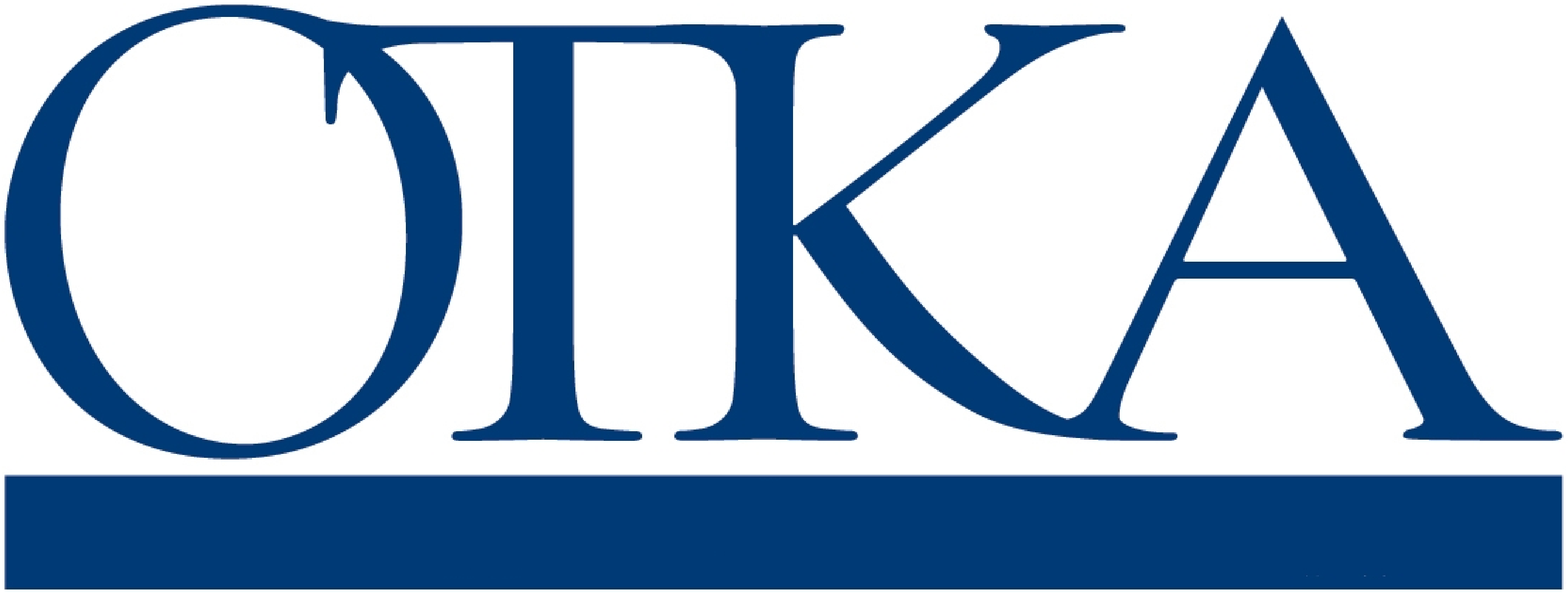}\tabularnewline
\end{tabular}}

\address{Department of Mathematics, Tampere University of Technology, PL 553,
33101 Tampere, Finland}

\email{sf@tut.fi}

\address{Department of Mathematics, Washington University in St.~Louis, St.~Louis,
MO, 63130}

\email{russw@math.wustl.edu}
\begin{abstract}
Rival and Zaguia showed that the antichain cutsets of a finite Boolean
lattice are exactly the level sets. We show that a similar characterization
of antichain cutsets holds for any strongly connected poset of locally
finite height. As a corollary, we characterize the antichain cutsets
in semimodular lattices, supersolvable lattices, Bruhat orders, locally
shellable lattices, and many more. We also consider a generalization
to strongly connected $d$-uniform hypergraphs.

\global\long\def\height{\operatorname{height}}
\vspace{-0.9cm}

\end{abstract}
\maketitle

\section{\label{sec:Introduction}Introduction}

An \emph{antichain cutset} in a poset $P$ is a set of elements that
intersects each maximal chain in exactly one element. For example,
in a graded poset, any \emph{level set} (consisting of all elements
of a specified rank) is an antichain cutset. Rival and Zaguia showed
\cite[Theorem 4]{Rival/Zaguia:1985} that for a finite Boolean lattice
the converse holds: every antichain cutset is a level set. Other papers
discussing antichain cutsets include \cite{Behrendt:1991,BoYu:1991,El-Zahar/Zaguia:1986,Ginsburg/Rival/Sands:1986,Rival/Zaguia:1987};
antichain cutsets also appear in Rota's well-known Crosscut Theorem
\cite{Bjorner:1995,Stern:1999}.

The purpose of the current paper is to generalize this converse result
of Rival and Zaguia to wide families of finite and infinite posets.
A poset $P$ is said to be \emph{strongly connected} if for any two
maximal chains $\mathbf{c}$ and $\mathbf{d}$ there is a sequence
of maximal chains 
\[
\mathbf{c}=\mathbf{c}^{(0)},\mathbf{c}^{(1)},\mathbf{c}^{(2)},\dots,\mathbf{c}^{(n)}=\mathbf{d}
\]
 such that the symmetric difference of $\mathbf{c}^{(i)}$ and $\mathbf{c}^{(i+1)}$
has cardinality two. Then:
\begin{thm}
\label{thm:MainThm-weak}If a poset $P$ is strongly connected with
distinct antichain cutsets $A$ and $B$, then $A$ and $B$ are disjoint.
\end{thm}
Theorem \ref{thm:MainThm-weak} will follow immediately from the somewhat
stronger Theorem \ref{thm:MainThm-strong}. We note that Theorem \ref{thm:MainThm-weak}
holds for any poset, with no assumption that the poset is discrete
or graded.

In the case of a discrete poset we can say more:
\begin{thm}
\label{thm:MainThm-weakgraded}If a discrete poset $P$ is strongly
connected, then the antichain cutsets of $P$ are exactly the level
sets.
\end{thm}
We contrast Theorem \ref{thm:MainThm-weakgraded} with the result
of Behrendt \cite{Behrendt:1991} that any lattice may be obtained
as the ``lattice of antichain cutsets'' of a poset with height 3.
In the situation of Theorem \ref{thm:MainThm-weakgraded}, the lattice
of antichain cutsets is a chain.\smallskip{}

The term ``strongly connected'' comes from a relationship with geometry,
and from the geometric combinatorics literature we obtain a large
list of examples:
\begin{prop}
\label{pro:ACClevelsetsExamples}Let $P$ be any locally finite height
semimodular lattice or supersolvable lattice, or a Bruhat order of
any Coxeter group, or any graded $EL$-shellable/shellable/Cohen-Macaulay
poset. Then the antichain cutsets of $P$ are exactly the level sets.
\end{prop}
\medskip{}
The paper is organized as follows. In the remainder of this section,
we introduce additional useful terminology for posets. In Section
\ref{sec:Proof-of-Main} we prove Theorems \ref{thm:MainThm-weak}
and \ref{thm:MainThm-weakgraded}. In Section \ref{sec:Connections}
we make the connection with geometric combinatorics, and prove Proposition
\ref{pro:ACClevelsetsExamples}. We close in Section \ref{sec:Generalization-to-hypergraphs}
by briefly sketching a generalization to uniform hypergraphs.

The paper incorporates and expands upon \cite{Foldes:2011UNP}.

\subsection{Terminology and notation}

For any poset $P$, the \emph{length} of a finite chain $\mathbf{c}$
is one less than the cardinality of $\mathbf{c}$. The \emph{height}
of $P$ is the supremum of lengths over all finite chains in $P$. 

Let $\mathcal{X}$ be some poset property, such as finiteness or strong
connectivity. We say that $P$ is \emph{locally $\mathcal{X}$} if
every interval $[a,b]$ in $P$ has property $\mathcal{X}$. For example,
a poset is \emph{locally strongly connected} if every interval is
strongly connected, and has \emph{locally finite height }if every
interval has finite height. 

We say that $P$ is \emph{pairwise-locally $\mathcal{X}$} if for
all $x,y\in P$ there is some interval containing $x$ and $y$ that
has property $\mathcal{X}$. It is obvious that if $\mathcal{X}$
is a property closed under taking subintervals, then locally $\mathcal{X}$
and pairwise-locally $\mathcal{X}$ are equivalent for lattices, but
this does not hold for general posets. We remark that pairwise-locally
$\mathcal{X}$ already requires that any $x$ and $y$ have some upper
and lower bound.

A poset is \emph{bounded} if it has a greatest element $\hat{1}$
and least element $\hat{0}$. If $x<y$, and if there is no $z$ with
$x<z<y$, then we write $x\lessdot y$, and say that $y$ \emph{covers}
$x$. If $P$ has a least element $\hat{0}$, then the \emph{atoms}
of $P$ are the elements covering $\hat{0}$.
\begin{example}
\label{exa:SConnNotLocally}Consider the poset $E$ consisting of
the natural numbers $\{1,2,3,4,5,30,25,200,300,600\}$, ordered by
divisibility. It is easy to verify that $E$ is strongly connected,
and since $E$ is bounded, it is also pairwise-locally strongly connected.
On the other hand, the interval $[1,200]$ is not strongly connected,
so $E$ is not locally strongly connected.
\end{example}
For any totally ordered set $R$, we define an \emph{$R$-grading}
or \emph{$R$-ranking} of a poset $P$ to be a map $\rho:P\rightarrow R$
such that $\rho$ restricts on each maximal chain of $P$ to an isomorphism
(onto $R$). In a poset with finite height $n-1$, the existence of
an $R$-grading (or more specifically an $[n]$-grading, where $[n]=\{1,\dots,n\}$)
is obviously equivalent to the usual notion of gradedness as defined
in e.g. \cite{Stanley:1997}. More generally, a similar equivalence
holds in any poset with all nonempty chains having a minimal element.
A \emph{level set} of a poset with an $R$-grading $\rho$ is any
non-empty set of the form $L_{r}=\{x\,:\,\rho(x)=r\}$. 

A poset is \emph{discrete} if every interval has a maximal chain of
finite length: thus, in a discrete poset $x<y$ if and only if there
is a chain of cover relations between the two. We caution that the
definition of discreteness for posets is not entirely consistent in
the literature, and that for example \cite{Stern:1999} defines discreteness
to mean locally finite height, a strictly stronger condition. This
inconsistency should not be confusing, as we will mostly consider
$R$-graded posets, where discreteness and locally finite height are
equivalent.

For other standard poset notation and terminology, we refer to \cite{Stanley:1997}
or \cite{Stern:1999}.

\section{Proof of Main Theorems\label{sec:Proof-of-Main}}

We prove a stronger variant of Theorem \ref{thm:MainThm-weak}, which
follows by adding a $\hat{0}$ and $\hat{1}$ element if necessary:
\begin{thm}
\label{thm:MainThm-strong}If $P$ is a pairwise-locally strongly
connected poset with antichain cutsets $A$ and $B$, then $A$ and
$B$ are disjoint.\end{thm}
\begin{proof}
Suppose by contradiction that $A$ and $B$ are antichain cutsets
with nontrivial intersection. If $y\in A\cap B$ and $z\in B\setminus A$,
then there is some strongly connected interval $I$ containing $y$
and $z$, and it follows easily that $A\cap I$ and $B\cap I$ are
antichain cutsets for $I$. 

Let $\mathcal{M}$ be the set of maximal chains on $I$ that intersect
$B$ at a member of $A\cap B$, and $\mathcal{N}$ be the set of maximal
chains on $I$ that intersect $B$ at a member of $B\setminus A$.
By construction both $\mathcal{M}$ and $\mathcal{N}$ are nonempty.
Applying the strongly connected property to $\mathbf{c}\in\mathcal{M}$
and $\mathbf{d}\in\mathcal{N}$ gives that there exist maximal chains
$\mathbf{m}\in\mathcal{M}$ and $\mathbf{n}\in\mathcal{N}$ such that
the symmetric difference of $\mathbf{m}$ and $\mathbf{n}$ is 2. 

Let $x$ be the unique element of $\mathbf{m}\cap A=\mathbf{m}\cap B$,
and let $a,b$ denote respectively the unique elements of $\mathbf{n}\cap A$
and $\mathbf{n}\cap B=\mathbf{n}\cap(B\setminus A)$. But then $\mathbf{m}\setminus\{x\}$
intersects neither $A$ nor $B$, hence all of $x,a,b$ are contained
in the symmetric difference of $\mathbf{m}$ and $\mathbf{n}$. It
follows that $a=b$, a contradiction.
\end{proof}
In the presence of an $R$-grading, we additionally have:
\begin{thm}
\label{thm:Main2LocallyConn}If $P$ is a pairwise-locally strongly
connected and $R$-graded poset (for some total order $R$), then
the antichain cutsets of $P$ are exactly the level sets.\end{thm}
\begin{proof}
The level sets partition $P$, hence every antichain cutset intersects
some level set. Since a level set is an antichain cutset (by definition
of $R$-grading), the result follows from Theorem \ref{thm:MainThm-strong}.
\end{proof}
To prove Theorem \ref{thm:MainThm-weakgraded} it remains to show
that every discrete strongly connected poset is $R$-graded for some
$R$. We start with a lemma:
\begin{lem}
\label{lem:PairwiseCompChains}If $\mathbf{m}$ and $\mathbf{n}$
are maximal chains in a strongly connected poset $P$, then $\mathbf{m}\setminus\mathbf{n}$
and $\mathbf{n}\setminus\mathbf{m}$ have the same finite cardinality.\end{lem}
\begin{proof}
By strong connectivity, $\mathbf{m}$ and $\mathbf{n}$ differ by
a finite number of exchanges; each exchange alters the cardinality
of $\vert\mathbf{m}\setminus\mathbf{n}\vert$ and $\vert\mathbf{n}\setminus\mathbf{m}\vert$
by the same number ($0$, $1$, or $-1$).\end{proof}
\begin{cor}
\label{cor:DiscreteSConnLFH}If $P$ is a discrete poset which is
strongly connected, locally strongly connected, or pairwise-locally
strongly connected, then every interval of $P$ is graded of finite
height.\end{cor}
\begin{lem}
\label{lem:LocallyGradingGivesGlobal}Let $P$ be a discrete and locally
graded poset such that every pair of elements $x,y\in P$ have an
upper bound and a lower bound in common. Then $P$ is $R$-graded
for some $R\subseteq\mathbb{Z}$.\end{lem}
\begin{proof}
Pick some $x_{0}\in P$ and define $\rho(x_{0})=0$. Then for any
$y\in P$, there is a lower bound $w$ and an upper bound $z$ of
$x_{0}$ and $y$. We let $\rho(y)=\height[w,y]-\height[w,x_{0}]$.
It is clear that $\rho$ is injective on each maximal chain of $P$,
and any two maximal chains are mapped by $\rho$ onto the same subset
of $\mathbb{Z}$ by the existence of common upper bounds.

It remains to check that $\rho$ is well-defined, i.e., that it doesn't
depend on the choice of $w$. But if we pick some other lower bound
$w'$ of $x_{0}$ and $y$, then there is a lower bound $v$ of $w$
and $w'$, and then $\rho$ restricts to a (well-defined) finite height
grading of $[v,z]$. \end{proof}
\begin{cor}
If $P$ is a discrete poset which is strongly connected or pairwise-locally
strongly connected, then $P$ is $R$-graded for some $R\subseteq\mathbb{Z}$.\end{cor}
\begin{proof}
For $P$ pairwise-locally strongly connected this is immediate by
Lemma \ref{lem:LocallyGradingGivesGlobal} with Corollary \ref{cor:DiscreteSConnLFH}.

For $P$ strongly connected, it is clear that all maximal chains are
isomorphic. If some pair of elements $x$ and $y$ fail to have a
common lower bound, then all maximal chains have a least element by
an easy application of Lemma \ref{lem:PairwiseCompChains}. Then $P\cup\{\hat{0}\}$
is discrete, and $\hat{0}$ is a lower bound for all elements. Similarly
for upper bounds. 

We augment $P$ by $\hat{0}$ and/or $\hat{1}$ if necessary (as in
the preceding paragraph) to obtain a discrete poset $\widehat{P}$
where every pair of elements has a common upper and lower bound. Since
adding/deleting $\hat{0}$ and/or $\hat{1}$ preserves strong connectivity,
we obtain from Lemma \ref{lem:LocallyGradingGivesGlobal} with Corollary
\ref{cor:DiscreteSConnLFH} that $\widehat{P}$ is $R$-graded. The
conclusion follows for $P$.
\end{proof}

\section{Relationships with geometric combinatorics\label{sec:Connections}}

\subsection{Locally connected posets and their order complexes\label{sub:OrderComplexes}}

The name ``strongly connected'' comes from a relationship with geometric
combinatorics. 

Recall that an \emph{(abstract) simplicial complex} is a family of
finite sets (called \emph{faces}) on a base set (called \emph{vertices})
which is closed under inclusion, and that any abstract simplicial
complex can be treated as a topological space (specifically, a cell
complex -- see e.g. \cite{Hatcher:2002}) by identifying each face
$\sigma$ having $(n+1)$ points with an $n$-dimensional simplex.
The \emph{dimension} of a simplicial complex is the supremum over
all faces of the dimension of this corresponding simplex. 

A finite dimensional simplicial complex is \emph{strongly connected}
if for every pair of \emph{facets} (maximal faces) $\sigma$ and $\tau$,
there is a sequence of facets between $\sigma$ and $\tau$ with each
adjacent pair having symmetric difference 2. For example, any simplicial
complex realizing a manifold is strongly connected.

Associated with any finite height poset $P$ is a simplicial complex
$\Delta(P)$ (the \emph{order complex}), where the vertices are the
elements of $P$, and the faces consist of all chains in $P$ \cite{Wachs:2007}.
The poset properties of $P$ are closely related to the topological
properties of $\Delta(P)$: for example, the Möbius number $\mu_{P}(x,y)$
of a finite interval $[x,y]$ is exactly the reduced Euler characteristic
$\tilde{\chi}(\Delta(x,y))$. We remark that bounded posets are contractible
(such a poset $P$ is a topological cone over $\Delta(P\setminus\{\hat{0}\})$),
so it is natural to consider open intervals or $\Delta(P\setminus\{\hat{0},\hat{1}\})$.

\medskip{}
In particular, it is clear that a finite height poset $P$ is strongly
connected in our sense if and only if $\Delta(P)$ is strongly connected
in the simplicial complex sense. 

\medskip{}

We say that a poset $P$ is \emph{connected }if there is a sequence
$x=x_{0}\leq x_{1}\geq x_{2}\leq\dots x_{n}=y$ between any two elements
$x$ and $y$. In the finite height case, this is equivalent to $\Delta(P)$
being path-connected; and in the discrete case it is equivalent to
the Hasse diagram being a connected graph. As before, a bounded poset
or closed interval is always connected, so to obtain non-trivial statements
we remove $\hat{0}$ and $\hat{1}$ or consider open intervals.

The following generalization of \cite[Proposition 11.7 and following]{Bjorner:1995}
(see also \cite[Lemma 4.2]{Klee:2009}) to infinite posets is proved
entirely similarly to the finite case:
\begin{lem}
\label{lem:LocalConnGivesStrongConn}If $P$ is a bounded poset of
finite height such that every open interval $(x,y)$ in $P$ of nonzero
height is connected, then $P$ is strongly connected.
\end{lem}
We see from Example \ref{exa:SConnNotLocally} that the sufficient
condition of Lemma \ref{lem:LocalConnGivesStrongConn} is not necessary.
We further comment that the result of Lemma \ref{lem:LocalConnGivesStrongConn}
fails to hold in posets with infinite height, as the following example
demonstrates:
\begin{example}
\label{exa:LocallySConnNotSConn}Consider the lattice $L=\left(\mathbb{N}\times\mathbb{N}\right)\cup\hat{1}$.
Every closed interval is isomophic to either $[n]\times[m]$ or to
$L$ itself, hence every open interval of nonzero height is connected.
But there is no strongly connected sequence interpolating between
the chain $\{(a,0)\,:\, a\in\mathbb{N}\}$ and $\{(0,b)\,:\, b\in\mathbb{N}\}$,
as every interval of the form $[(a_{0},0),(a_{n},0)]$ is a chain.

We comment that since $\mathbb{N}\times\mathbb{N}$ is locally strongly
connected and hence pairwise-locally strongly connected, the antichain
cutsets of $L$ nonetheless consist exactly of the level sets.\end{example}
\begin{rem}
Theorem \ref{thm:MainThm-strong} gives a connection between the combinatorics
of antichain cutsets in $P$ and the topology of $\Delta(P)$. Bell
and Ginsburg \cite{Bell/Ginsburg:1984} (see also \cite{Ginsburg/Rival/Sands:1986})
gave another topological condition, showing that a certain space $M(P)$
is compact if and only if every $x\in P$ is in some finite antichain
cutset. We are not aware of any connection between the order complex
and the topology considered by Bell and Ginsburg.
\end{rem}

\subsection{$EL$-labelings}

An \emph{edge labeling} of a locally finite height poset $P$ is a
map $\lambda$ from the cover relations of $P$ (i.e., the edges of
the Hasse diagram) to some partially ordered set $\Lambda$, usually
the integers. Edge labelings associate a word of elements from $\Lambda$
with every maximal chain on an interval of $P$, and we order such
chains lexicographically by their associated words. We say that $\lambda$
has an \emph{ascent} at $y$ if $\lambda(x\lessdot y)\leq\lambda(y\lessdot z)$,
and a \emph{descent} otherwise; a chain is \emph{ascending} if every
element other than the top and bottom is an ascent.

An \emph{$EL$-labeling} of a locally finite height poset $P$ is
an edge labeling such that on every interval $[x,y]$ of $P$:
\begin{enumerate}
\item There is a unique ascending maximal chain $\mathbf{a}^{[x,y]}$.
\item The ascending chain $\mathbf{a}^{[x,y]}$ lexicographically precedes
all other maximal chains on $[x,y]$.
\end{enumerate}
When $P$ is infinite we also require the following, which we note
to be automatic in the finite case:
\begin{enumerate}
\item [(3)] The lexicographic order of maximal chains on $[x,y]$ has a
linear extension which is a well-order.\end{enumerate}
\begin{example}
In a finite height semimodular lattice $L$, any (strict) well-ordering
$\prec$ of the join irreducible elements such that $x<y$ implies
$x\prec y$ induces an $EL$-labeling with label set $L$. The labeling
assigns to $x\lessdot y$ the first join irreducible $z$ such that
$x\vee z=y$ \cite[Proposition 2.2]{Stanley:1974}. 
\end{example}
A precursor to $EL$-labelings was introduced by Stanley \cite{Stanley:1974}
as a purely combinatorial description of some similar behavior between
geometric lattices and subgroup lattices of supersolvable groups,
and the definition as above was made and studied by Björner \cite{Bjorner:1980},
and Björner and Wachs \cite{Bjorner/Wachs:1996}. The existence of
an $EL$-labeling has strong consequences for the topology of $\Delta(P)$,
and we recommend \cite{Wachs:2007} for a highly readable account
of this. 

One such consequence is the following:
\begin{lem}
\label{lem:EL-shellableStrConn}If $P$ is a bounded poset which is
graded of finite height and admits an $EL$-labeling, then $P$ is
strongly connected.\end{lem}
\begin{proof}
It suffices to show that there is a sequence as in the definition
of strong connectivity from any chain $\mathbf{c}$ to the unique
ascending chain $\mathbf{a}^{[\hat{0},\hat{1}]}$. Suppose that $\mathbf{c}=\{\hat{0}=y_{0}\lessdot y_{1}\lessdot\dots\lessdot y_{n}=\hat{1}\}$
has a descent at $y_{i}$. Then applying the $EL$-labeling property
on $[y_{i-1},y_{i+1}]$ allows us to replace $y_{i}$ with a $y_{i}'$
so that $\mathbf{c}'=(\mathbf{c}\setminus y_{i})\cup y_{i}'$ has
an ascent at $y_{i}'$, and we notice that $\mathbf{c}$ and $\mathbf{c}'$
have symmetric difference two. Since $\mathbf{c}'$ lexicographically
precedes $\mathbf{c}$, repeating this process leaves us with $\mathbf{a}^{[\hat{0},\hat{1}]}$
after a finite number of steps (by the well-ordering condition).\end{proof}
\begin{cor}
If $P$ is a locally finite height $R$-graded lattice with an $EL$-labeling,
then the antichain cutsets of $P$ are exactly the level sets.\end{cor}
\begin{proof}
Apply Lemma \ref{lem:EL-shellableStrConn}, then Theorem \ref{thm:Main2LocallyConn}.\end{proof}
\begin{rem}
Although the proofs become less elementary, the consequence of Lemma
\ref{lem:EL-shellableStrConn} holds more generally for the class
of Cohen-Macaulay posets, as discussed below in Section \ref{sub:Shellable-and-Cohen-Macaulay}. 
\end{rem}

\subsection{Shellable and Cohen-Macaulay complexes \label{sub:Shellable-and-Cohen-Macaulay}}

We very briefly discuss some closely related properties implied by
the existence of an $EL$-labeling. For further background, refer
to \cite{Wachs:2007} for finite posets, and \cite[especially Remark 4.21]{Bjorner:1984}
for the infinite (but finite height) case. Another related work \cite{Aramova/Herzog/Hibi:2002}
considers locally finite shellable posets from an algebraic point
of view.

A finite height poset is \emph{shellable} if there is a well-ordering
$\prec$ (called a \emph{shelling}) of the maximal chains of $P$
such that if $\mathbf{c}\prec\mathbf{d}$ then there is some $\mathbf{c}'\prec\mathbf{d}$
with $\mathbf{c}\cap\mathbf{d}\subseteq\mathbf{c}'\cap\mathbf{d}=\mathbf{d}\setminus\{x\}$
for some $x$. Every interval of a shellable poset is also shellable.
For example, if $P$ is a bounded poset of finite height admitting
an $EL$-labeling, then (a well-ordered linear extension of) the lexicographic
order is a shelling. As a result, posets with an $EL$-labeling are
often referred to as \emph{$EL$-shellable.} The proof of Lemma \ref{lem:EL-shellableStrConn}
adapts straightforwardly to shellable posets.

A finite height poset is \emph{Cohen-Macaulay} if it obeys a certain
technical condition involving vanishing homology groups in low dimension
for every interval -- the details will not be important here, except
to note that every graded shellable poset is Cohen-Macaulay. A Cohen-Macaulay
poset has the property that every open interval of nonzero height
is connected, hence by Lemma \ref{lem:LocalConnGivesStrongConn} a
Cohen-Macaulay poset is strongly connected.

We summarize the chain of implications for a finite height graded
poset:
\begin{align*}
\mbox{bounded + admits }EL\mbox{-labeling}\implies\mbox{shellable}\implies & \mbox{Cohen-Macaulay}\\
\implies & \mbox{strongly connected},
\end{align*}
and comment that each implication is strict.
\begin{cor}
If $P$ is a Cohen-Macaulay poset (more generally, a pairwise-locally
Cohen-Macaulay poset), then the antichain cutsets of $P$ are exactly
the level sets.
\end{cor}

\subsection{Examples}

There is a vast literature on finite or locally finite posets and
lattices with $EL$-labelings and/or shellings, as we touch upon in
Proposition \ref{pro:ACClevelsetsExamples}.
\begin{example}[Discrete semimodular lattices]
 As previously mentioned, any finite height semimodular lattice admits
an $EL$-labeling. Thus, a discrete (equivalently locally finite height)
semimodular lattice has an $EL$-labeling in every interval, so is
locally strongly connected. As such a lattice is well-known to be
$R$-graded (for some $R\subseteq\mathbb{Z}$), the antichain cutsets
of a discrete semimodular lattice are exactly the level sets. Thus
the result of \cite{Foldes:2011UNP} is recovered.

As a special case, the lattice of flats of any matroid has antichain
cutsets consisting exactly of its level sets.
\end{example}
The \emph{subgroup lattice} $L(G)$ of a group $G$ consists of all
subgroups of $G$, ordered by inclusion.
\begin{example}[Supersolvable lattices]
 A finite height lattice is \emph{supersolvable} if it has height
$n$, and has an $EL$-labeling which labels every maximal chain with
a permutation of $[n]$ \cite{Stanley:1972,McNamara/Thomas:2006}.
It follows that a supersolvable lattice is necessarily graded, and
so has antichain cutsets consisting exactly of the level sets. Examples
of supersolvable lattices include subgroup lattices of finite supersolvable
groups, as well as partition lattices \cite{Stanley:1972}.
\end{example}
In particular, the subspace lattice of a finite vector space is both
(semi-)modular and supersolvable. Hence we immediately obtain a $q$-analogue
of \cite[Theorem 4]{Rival/Zaguia:1985}: the antichain cutsets of
the subspace lattice of any finite vector space are exactly the level
sets.

We also obtain that the subgroup lattice of $\mathbb{Z}$ has antichain
cutsets consisting exactly of its level sets, as $L(\mathbb{Z})$
is $(\mathbb{Z}^{-}\cup\hat{0})$-graded and every interval of $L(\mathbb{Z})\setminus\{\hat{0}\}$
is finite and supersolvable (and indeed distributive). 
\begin{example}[Bruhat order]
 Let $W$ be a Coxeter group (see \cite{Bjorner/Brenti:2005}) generated
by involutions $S$, and let $\ell(w)$ be the minimal length word
expressing $w\in W$ as the product of generators from $S$. The \emph{Bruhat
order} associated with $(W,S)$ has a cover relation between $w$
and $wt$ if $t$ is conjugate to some $s\in S$ and $\ell(wt)=\ell(w)+1$.
Bruhat orders have locally finite height \cite[Corollary 2.2.4]{Bjorner/Brenti:2005},
are locally and pairwise-locally shellable \cite[Proposition 2.2.9 and Corollary 2.7.5]{Bjorner/Brenti:2005},
and are graded by $\ell$, hence their antichain cutsets are exactly
the level sets.
\end{example}

\section{Generalization to hypergraphs\label{sec:Generalization-to-hypergraphs}}

We remark briefly that our proof of Theorem \ref{thm:MainThm-weak}
has a generalization to $d$-uniform hypergraphs. Recall that a \emph{transversal}
of a hypergraph is a subset $S$ of vertices that has nonempty intersection
with every edge. General background on transversals can be found in
\cite{Berge:1989}. An \emph{exact transversal} is a transversal that
meets each edge in exactly one vertex. Thus, an antichain cutset for
a poset $P$ is an exact transversal for the hypergraph with vertex
set $P$ and edge set consisting of the maximal chains of $P$. 

We define a $d$-uniform hypergraph $\mathcal{H}$ to be \emph{strongly
connected} if for every pair $e$ and $f$ of edges of $\mathcal{H}$
there is a sequence 
\[
e=e^{(0)},e^{(1)},\dots,e^{(n)}=f
\]
 such that every adjacent pair has symmetric difference 2. The simplicial
complex $\left\langle \mathcal{H}\right\rangle $ \emph{generated
by $\mathcal{H}$ }has faces consisting of all subsets of any edge
of $\mathcal{H}$. It is clear that $\mathcal{H}$ is strongly connected
if and only if $\left\langle \mathcal{H}\right\rangle $ is strongly
connected in the simplicial complex sense discussed in Section \ref{sub:OrderComplexes}.

The argument leading to Theorem \ref{thm:MainThm-weak} translates
straightforwardly to:
\begin{thm}
\label{thm:HypGen}If $\mathcal{H}$ is a strongly connected $d$-uniform
hypergraph, then the exact transversals of $\mathcal{H}$ are pairwise
disjoint.
\end{thm}
The definitions for shellability and Cohen-Macaulay that we presented
in Section \ref{sub:Shellable-and-Cohen-Macaulay} are specialized
from more general definitions of simplicial complexes, and strong
connectivity still holds for such.
\begin{cor}
If $\left\langle \mathcal{H}\right\rangle $ is pure shellable or
Cohen-Macaulay, then the exact transversals of $\mathcal{H}$ are
pairwise disjoint.
\end{cor}
A $d$-uniform hypergraph is \emph{balanced} if there is a $d$-coloring
of the vertices such that no edge contains two vertices with the same
color. Balanced uniform hypergraphs generalize (the set of maximal
chains of) graded posets. There is some literature on balanced Cohen-Macaulay
simplicial complexes \cite[Chapter III.4 and references]{Stanley:1996}. 

The same proof as that of Theorem \ref{thm:Main2LocallyConn} yields:
\begin{thm}
\label{thm:HypGenBalanced}If $\mathcal{H}$ is a balanced strongly
connected $d$-uniform hypergraph, then the exact transversal sets
of $\mathcal{H}$ consist exactly of the color classes of vertices.
\end{thm}
We comment that Theorems \ref{thm:HypGenBalanced} and \ref{thm:MainThm-weakgraded}
can be seen as a generalization of the fact that a connected bipartite
graph has a unique 2-coloring (up to renaming the colors).
\begin{example}
Consider the \emph{chessboard complex}: the simplicial complex with
vertex set the squares of an $a\times b$ chessboard ($a\leq b$),
and faces consisting of ``non-taking rook configurations'', i.e.,
sets of squares with no two on the same row or column. For any $a,b$
this simplicial complex is clearly $(a-1)$-dimensional. Let $\mathcal{C}_{a,b}$
be the hypergraph of maximal faces, that is, with edges consisting
of sets of $a$ non-taking rooks.

We notice that $\mathcal{C}_{a,b}$ is balanced, since every edge
contains a rook from each row. If $a<b$, then it is straightforward
to show by exchanging rooks that $\mathcal{C}_{a,b}$ is strongly
connected (this was essentially observed in \cite{Bjorner/Lovasz/Vrecica/Zivaljevic:1994}).
It follows that if $a<b$ then the exact transversals for $\mathcal{C}_{a,b}$
are in bijection with the rows of the chessboard, and consist of $b$
vertices corresponding to the squares of the given row.

In contrast, we notice that $\mathcal{C}_{a,a}$ does not have disjoint
exact transversals, since we have exact transversals in correspondence
with the rows of the chessboard, and (by transposition) additional
exact transversals in correspondence with the columns. 
\end{example}
\bibliographystyle{hamsplain}
\bibliography{6_Users_russw_Documents_Research_Master}

\end{document}